\documentclass[12 pt,twoside]{amsart}

\usepackage{amsmath,amsfonts,amsthm,amssymb}

\newtheorem{theorem}{Theorem}[section]
\newtheorem{lemma}[theorem]{Lemma}
\newtheorem{proposition}[theorem]{Proposition}

\theoremstyle{definition}

\theoremstyle{remark}

\numberwithin{equation}{section}

\begin{document}

\title [Tuples of complex triangular Toeplitz matrices]{Dynamics of tuples of complex upper triangular Toeplitz matrices}

\author[G. Costakis]{G. Costakis}
\address{Department of Mathematics, University of Crete, Knossos Avenue, GR-714 09 Heraklion, Crete, Greece}
\email{costakis@math.uoc.gr}
\thanks{}

\author[D. Hadjiloucas]{D. Hadjiloucas}
\address{The School of Sciences, European University Cyprus, 6 Diogenes Street, Engomi, P.O.Box 22006, 1516 Nicosia, Cyprus}
\email{d.hadjiloucas@euc.ac.cy}
\thanks{}

\author[A. Manoussos]{A. Manoussos}
\address{Fakult\"{a}t f\"{u}r Mathematik, SFB 701, Universit\"{a}t Bielefeld, Postfach 100131, D-33501 Bielefeld, Germany}
\email{amanouss@math.uni-bielefeld.de}
\thanks{The second author would like to thank European University Cyprus for the financial
support received in carrying out part of this work under the internal research grant for the project titled
"Dynamics of Tuples of Matrices". During this research the third author was fully supported by SFB 701
``Spektrale Strukturen und Topologische Methoden in der Mathematik" at the University of Bielefeld, Germany. He
would also like to express his gratitude to Professor H. Abels for his support.}

\date{}



\begin{abstract}
Based on a question raised by N. Feldman we discuss the dynamics of tuples of upper triangular Toeplitz matrices over $\mathbb{C}$. Some open problems are posed.
\end{abstract}
\maketitle

\section{Introduction}
In the last twenty years the subject of hypercyclicity has been studied intensively. Hypercyclicity is a phenomenon occurring only in infinite dimensional topological
vector spaces, see for example \cite{BM}. That is, if for a continuous linear operator $T$ acting on a topological vector space $X$ there exists a vector $x\in X$
whose orbit under $T$ is dense in $X$, i.e. $$\overline{Orb(T,x)}=\overline{ \{ x, Tx, T^2x, \dots \} }=X$$ then $X$ is infinite dimensional. However, in
\cite{Feldman2}, \cite{Feldman}, Feldman showed that hypercyclic tuples of matrices may exist in finite dimensional vector spaces. Namely, there exist $n\times n$
diagonal matrices $T_1, \ldots ,T_{n+1}$ with complex entries and vector $x\in \mathbb{C}^n$ such that the set
$$ \{ T_1^{k_1}\cdots T_{n+1}^{k_{n+1}}x: k_1,\dots ,k_{n+1}\in \mathbb{N}\cup \{ 0\} \}$$ is dense in
$\mathbb{C}^n$. Such vectors $x\in \mathbb{C}^n$ are called hypercyclic for the tuple $(T_1, \ldots ,T_{n+1})$. In addition, Feldman proved that no $n$-tuple of
diagonal $n\times n$ matrices with complex entries is hypercyclic on $\mathbb{C}^n$. Complementing Feldman's result we showed in \cite{CoHaMa} that there exist
hypercyclic $k$-tuples of non simultaneously diagonalizable (commuting) $n\times n$ matrices with $k>n$ on $\mathbb{C}^n$. Hence the following question due to Feldman
\cite{Feldman2} naturally arises.

\medskip
\noindent \textbf{Question} \textit{Is there a hypercyclic (non simultaneously diagonalizable) $n$-tuple on $\mathbb{C}^n$?}

\medskip
The purpose of the present work is twofold. On one hand we give a negative answer to the aforementioned question in the case where at least one of the matrices in the
tuple is cyclic. On the other hand we pose some (related) open problems on the dynamics of tuples of upper triangular Toeplitz matrices as well as raise questions
concerning minimal actions. For simplicity we demonstrate our method, which is completely elementary, for the cases $n=2,3,4$. Recently, Ayadi in \cite{A} has
answered Feldman's question in full generality using more sophisticated tools than ours, such as his previous very interesting work on abelian subgroups of
$GL(n,\mathbb{C})$, see \cite{AM1}, \cite{AM2}. We would also like to mention that H. Abels and the third author, among other things, were able to find the minimal
number of matrices either on $\mathbb{C}^n$ or $\mathbb{R}^n$ (the real case is much harder) that form a hypercyclic tuple, using algebraic group theory, see
\cite{AM}. As a result, they provide yet another approach to Feldman's question.

For an account of results on hypercyclicity we refer to the survey articles \cite{GE}, \cite{GE2}, \cite{GE3} and the recent books \cite{BM}, \cite{GP}. For results
on the dynamics of tuples of matrices see \cite{AM}, \cite{CoHaMa}, \cite{Feldman1}, \cite{Feldman}, \cite{J1}, \cite{J2}, \cite{Ker}.

\section{A reduction to Toeplitz matrices}
Let us first fix some notation. For $m\in \mathbb{N}$ and $\lambda \in \mathbb{C}$ the symbol $J_m(\lambda )$ stands for a $m\times m$ matrix in Jordan form with
eigenvalue $\lambda $. Let $A$ be a $n\times n$ matrix with eigenvalues $\lambda_1,\ldots ,\lambda_k$ and multiplicities $n_1, \ldots ,n_k$ respectively. Then $A$ is
similar to a matrix $B$, where $B$ is the standard Jordan canonical form of $A$. Up to permutation, $B$ has the unique structure
$$B=J_{n_1}(\lambda_1)\oplus \ldots \oplus J_{n_k}(\lambda_k).$$
The next two lemmata are well known, with easy proofs, so we just state them for completeness.
\begin{lemma} \label{commute1}
Let $J,A$ be $n\times n$ commuting matrices over $\mathbb{C}$ such that $J=J_n(\lambda )$, $\lambda \in \mathbb{C}$. Then
\[
A=\left(
\begin{array}{llllcl}
a_{1} & a_{2} & a_{3} & a_{4} & \cdots & a_{n}\\
0 & a_{1} & a_{2} & a_{3} & \cdots & a_{n-1}\\
0 & 0 & a_{1} & a_{2} & \cdots & a_{n-2}\\
0 & 0 & 0 & a_{1} & \cdots & a_{n-3}\\
\vdots & \vdots & \vdots & & \ddots & \vdots\\
0& 0& 0& 0 & \cdots & a_{1}
\end{array} \right),
\]
where $a_1,\cdots ,a_n \in \mathbb{C}$.
\end{lemma}
Observe that the matrix $A$ in the above lemma is a (triangular) Toeplitz matrix.
\begin{lemma} \label{commute2}
Let $A,B$ be $n\times n$ commuting matrices over $\mathbb{C}$ such that $B$ is in a Jordan canonical form, i.e. $B=J_{n_1}(\lambda_1)\oplus \ldots \oplus
J_{n_k}(\lambda_k)$, for certain $\lambda_1, \ldots ,\lambda_k\in \mathbb{C}$ and $n_1+\ldots +n_k=n$ and assume in addition that $\lambda_m\neq \lambda_j$ for $m\neq
j$. Then $A$ has the form $A=A_1\oplus \ldots \oplus A_k$, where $A_j$ is a $n_j\times n_j$ (triangular) Toeplitz matrix as in Lemma \ref{commute1}.
\end{lemma}
The above lemmata point out that in order to understand the dynamics of tuples of commuting matrices for which at least one of them is cyclic, it suffices to know how
the iterates of (triangular) Toeplitz matrices behave. Let us briefly explain why. Suppose that $B_1, \ldots ,B_n$ are complex $n\times n$ commuting matrices and
assume that one of them is cyclic, say $B_1$. Then there is a similarity transformation matrix $P$ such that $P^{-1}B_1P=J_{n_1}(\lambda_1)\oplus \ldots \oplus
J_{n_k}(\lambda_k)$, for certain $\lambda_1, \ldots ,\lambda_k\in \mathbb{C}$, $n_1+\ldots +n_k=n$ and $\lambda_m\neq \lambda_j$ for $m\neq j$. In other words
$P^{-1}B_1P$ is just the Jordan canonical form of $B_1$. Observe now that the matrices $P^{-1}B_1P, \ldots ,P^{-1}B_nP$ are pairwise commuting and by Lemmata
\ref{commute1}, \ref{commute2} we have $P^{-1}B_jP=A_{j,1}\oplus \ldots \oplus A_{j,k}$ for $j=2,\ldots ,n$ and each $A_{j,\rho }$ is a $n_j\times n_j$ (triangular)
Toeplitz matrix, $ j=2,\ldots ,n$, $\rho =1,\ldots ,k$. It is easy to check that hypercyclicity is preserved under similarity. Therefore, we conclude that the tuple
$$(B_1,\ldots ,B_n)$$ is hypercyclic if and only if the tuple
$$(J_{n_1}(\lambda_1)\oplus \ldots \oplus J_{n_k}(\lambda_k), A_{2,1}\oplus \ldots \oplus A_{2,k}, \ldots , A_{n,1}\oplus \ldots \oplus A_{n,k} )$$ is hypercyclic.
It is also useful to observe that
$$ (A_{j,1}\oplus \ldots \oplus A_{j,k})^m= A_{j,1}^m\oplus \ldots \oplus A_{j,k}^m, \,\,\, m\in \mathbb{N} .$$
Hence, Toeplitz matrices come naturally into play when dealing with dynamics of commuting matrices for which at least one of them is cyclic. We shall prove the
following theorem.
\begin{theorem}
Suppose that $B_1,\ldots ,B_n$ are pairwise commuting complex $n\times n$ matrices so that $B_j$ is cyclic for some $j\in \{1,\ldots ,n\}$. Then the tuple $(B_1,
\ldots ,B_n)$ is not hypercyclic on $\mathbb{C}^n$.
\end{theorem}

\section{Iterates of Toeplitz matrices}

Let $n\in\mathbb{N}$ and consider the $n\times n$ Toeplitz matrices
\[
A_i=\left(
\begin{array}{llllcl}
a_{i,1} & a_{i,2} & a_{i,3} & a_{i,4} & \cdots & a_{i,n}\\
0 & a_{i,1} & a_{i,2} & a_{i,3} & \cdots & a_{i,n-1}\\
0 & 0 & a_{i,1} & a_{i,2} & \cdots & a_{i,n-2}\\
0 & 0 & 0 & a_{i,1} & \cdots & a_{i,n-3}\\
\vdots & \vdots & \vdots & & \ddots & \vdots\\
0& 0& 0& 0 & \cdots & a_{i,1}
\end{array} \right)
\]
for $i=1,2,\ldots,m$. We first begin by illustrating how to determine the elements of $A_i^k$ for some power $k$ without needing to go through the laborious process
of raising $A_i$ to the power $k$. Define the $n\times n$ nilpotent shift matrices $U_p$ for $p=0,1,2,\ldots$ by setting
\[
u_{i,j}=\left\{
\begin{array}{cc}
1, & j=i+p\\
0, & \textrm{otherwise}.
\end{array}\right.
\]
Observe that $U_0$ is the $n\times n$ identity matrix and $U_p=U_1^p$ which implies that $U_p U_q=U_q U_p = U_{p+q}$ for any $p,q\in\{0,1,2,\ldots\}$. Also, $U_p$ is
the $n\times n$ zero matrix for any $p\geq n$. As a result we have that
\[
A_i=\sum_{j=1}^n a_{i,j} U_{j-1}.
\]
It now follows from the multinomial theorem that
\begin{eqnarray*}
A_i^k &=& \left(\sum_{j=1}^n a_{i,j} U_{j-1}\right)^k\\
      &=& \sum_ {\substack{k_{i,1},\ldots, k_{i,n} \\
      k_{i,1}+ \ldots + k_{i,n}=k}} \binom{k}{k_{i,1},\ldots, k_{i,n}} a_{i,1}^{k_{i,1}} \ldots a_{i,n}^{k_{i,n}}\; U_1^{k_{i,2}+2 k_{i,3}+\ldots +(n-1) k_{i,n}}
\end{eqnarray*}
where
\[
\binom{k}{k_{i,1}, k_{i,2}, \ldots, k_{i,n}}=\frac{k!}{k_{i,1}! k_{i,2}!\ldots k_{i,n}!}
\]
is the multinomial coefficient. This example illustrates the ease with which one can compute the entries of the matrix where now attention is paid on the set of
configurations which make the expression $k_{i,2}+2 k_{i,3}+\ldots +(n-1) k_{i,n}$ equal to a given integer which by nature is more of a combinatorial problem.

Let us now turn to the product $A_1^{k_1}\ldots A_m^{k_m}$ which is
\begin{eqnarray*}
\sum_{\substack{k_{1,1},\ldots, k_{1,n} \\ k_{1,1}+\ldots + k_{1,n}=k_1}}\ldots & &\sum_{\substack{k_{m,1}\ldots k_{m,n} \\k_{m,1}+\ldots + k_{m,n}=k_m}}
\prod_{j=1}^{m}\binom{k_j}{k_{j,1},\ldots, k_{j,n}} a_{j,1}^{k_{j,1}}\ldots a_{j,n}^{k_{j,n}}\times\\
& & \times U_1^{k_{1,2}+\ldots + k_{m,2}+2(k_{1,3}+\ldots+k_{m,3})+\ldots+(n-1)(k_{1,n}+\ldots+k_{m,n})}.
\end{eqnarray*}
The only difficulty in such a computation is on determining the set of configurations which make the expression
\[
k_{1,2}+\ldots + k_{m,2}+2(k_{1,3}+\ldots+k_{m,3})+\ldots+(n-1)(k_{1,n}+\ldots+k_{m,n})
\]
equal to a given integer.

\section{2$\times$2 matrices}
For this section let $B_1$, $B_2$ be $2\times 2$ non-simultaneously diagonalizable commuting matrices such that $B_1$ is cyclic. Hence $B_1$ has only one eigenvalue
and the same is true for $B_2$ as well. Then from the analysis in section 2, there exists a $2\times 2$ similarity transformation matrix $P$ such that both
$A_1=P^{-1}B_1P$, $A_2=P^{-1}B_2P$ are triangular Toeplitz matrices, i.e.
\[
A_1=\left(
\begin{array} {ll}
a_{1,1} & a_{1,2}\\
0 & a_{1,1}
\end{array}
\right) , A_2=\left(
\begin{array} {ll}
a_{2,1} & a_{2,2}\\
0 & a_{2,1}
\end{array}
\right).
\]
Then for a multi-index $k=(k_1,k_2)$, $k_1,k_2\in \mathbb{N}$ define
\[
A(k)=\left(
\begin{array} {ll}
a_1(k) & a_2(k)\\
0 & a_1(k)
\end{array}
\right) :=A_1^{k_1}A_2^{k_2}.
\]

A direct computation shows the following, provided that $a_{1,1}\neq 0$, $a_{2,1}\neq 0$.
\begin{lemma} \label{coef2}
\begin{eqnarray*}
a_1(k) &=& a_{1,1}^{k_1} a_{2,1}^{k_2}\\
a_2(k) &=& a_1(k) \sum_{j=1}^{2} k_j\frac{a_{j,2}}{a_{j,1}}
\end{eqnarray*}
\end{lemma}

\begin{lemma} \label{equivhyp2}
The tuple $(A_1,A_2)$ is hypercyclic if and only if the set
\[
\left\{\left(
\begin{array}{c}
\sum_{j=1}^{2} k_j\frac{a_{j,2}}{a_{j,1}}\\
a_{1,1}^{k_1} a_{2,1}^{k_2}
\end{array}
\right) : k_1, k_2\in \mathbb{N}\right\}
\]
is dense in $\mathbb{C}^2$.
\end{lemma}
\begin{proof}
Assume that $(A_1,A_2)$ is hypercyclic with hypercyclic vector $(z_1,z_2)\in \mathbb{C}^2$. Then, we necessarily have $a_{1,1}\neq 0$ and $a_{2,1}\neq 0$ and the set
$$ \{ (a_1(k)z_1+a_2(k)z_2, a_1(k)z_2) : k=(k_1,k_2)\in \mathbb{N}\times \mathbb{N} \} $$ is dense in $\mathbb{C}^2$. From the last fact, it follows that $z_2\neq 0$. Take $(w_1,w_2)\in \mathbb{C}^2$. From our hypothesis and Lemma \ref{coef2} there exist sequences $\{ k_1(n)\}$, $\{ k_2(n) \}$ of positive integers such that
$$ a_1(k(n))z_2\to w_2z_2$$
$$a_1(k(n))z_1+a_2(k(n))z_2\to w_2z_1+w_1w_2z_2 ,$$
where $k(n)=(k_1(n),k_2(n))$. Since $z_2\neq 0$, from the above we get
$$ a_1(k(n))\to w_2 \quad \textrm{and} \quad  \sum_{j=1}^{2} k_j(n)\frac{a_{j,2}}{a_{j,1}} \to w_1.$$ For the converse implication one works in a similar manner.
\end{proof}

Once we have at our disposal the above lemma, one can follow the argument from \cite{Feldman} in order to conclude that the tuple $(A_1,A_2)$ is not hypercyclic.
\begin{proposition}
The tuple $(A_1, A_2)$ is not hypercyclic.
\end{proposition}
\begin{proof}
Suppose on the contrary that $(A_1,A_2)$ is hypercyclic. By Lemma \ref{equivhyp2} it follows that the set
\[
\left\{\left(
\begin{array}{c}
\sum_{j=1}^{2} k_j\frac{a_{j,2}}{a_{j,1}}\\
a_{1,1}^{k_1} a_{2,1}^{k_2}
\end{array}
\right) :k_1, k_2\in \mathbb{N}\right\}
\]
is dense in $\mathbb{C}^2$ and hence the set
\[
\left\{\left(
\begin{array}{c}
\left( e^{\frac{a_{1,2}}{a_{1,1}}} \right)^{k_1} \left( e^{\frac{a_{2,2}}{a_{2,1}}} \right)^{k_2}\\
a_{1,1}^{k_1} a_{2,1}^{k_2}
\end{array}
\right) :k_1, k_2\in \mathbb{N}\right\}
\]
is dense in $\mathbb{C}^2$. The last implies that the tuple of diagonal matrices
\[ \left(
\left(
\begin{array} {ll}
e^{\frac{a_{1,2}}{a_{1,1}}} & 0\\
0 & a_{1,1}
\end{array}
\right) , \left(
\begin{array} {ll}
e^{\frac{a_{2,2}}{a_{2,1}}} & 0\\
0 & a_{2,1}
\end{array}
\right) \right)
\]
is hypercyclic which is a contradiction from the results in \cite{Feldman}.
\end{proof}
From the discussion in section 2 and the proposition above we conclude that the tuple $(B_1,B_2)$ is not hypercyclic.

\section{3$\times$3 matrices}
In order to simplify things, assume that $B_1$ is a $3\times 3$ cyclic matrix having only one eigenvalue $\lambda$, i.e $B_1$ is similar to a $3\times 3$ matrix in
Jordan form, say $A_1=J_3(\lambda)$ through a similarity matrix $P$. Let us consider the more general situation where $A_1$ is a (triangular) Toeplitz matrix. To
illustrate how making use of the computation in section 3 can turn a purely algebraic problem (that of determining the elements of $A^k$) into a more or less
combinatorial one, consider raising to the power $k$ the $3\times 3$ matrix
\[
A_1=\left(
\begin{array} {lll}
a_{1,1} & a_{1,2} & a_{1,3}\\
0 & a_{1,1} & a_{1,2}\\
0 & 0 & a_{1,1}
\end{array}
\right).
\]
From the multinomial theorem
\[
A_1^k = \sum_{\substack{k_{1,1},k_{1,2},k_{1,3} \\k_{1,1}+k_{1,2}+ k_{1,3}=k}} \binom{k}{k_{1,1},k_{1,2}, k_{1,3}} a_{1,1}^{k_{1,1}} a_{1,2}^{k_{1,2}}
a_{1,3}^{k_{1,3}}\; U_1^{k_{1,2}+2 k_{1,3}}
\]
so the product matrix is of the form
\[
C=\left(
\begin{array} {lll}
c_{1,1} & c_{1,2} & c_{1,3}\\
0 & c_{1,1} & c_{1,2}\\
0 & 0 & c_{1,1}
\end{array}
\right).
\]
To determine $c_{1,1}$ we require $k_{1,2}+2 k_{1,3}=0$ which implies that $k_{1,2}=k_{1,3}=0$ and so $k_{1,1}=k$. As a result,
\[
c_{1,1} = \binom{k}{k,0,0} a_{1,1}^k a_{1,2}^0 a_{1,3}^0 =a_{1,1}^k.
\]
To determine $c_{1,2}$ we require $k_{1,2}+2 k_{1,3}=1$ which implies that $k_{1,3}=0$ and $k_{1,2}=1$. As a result,
\[
c_{1,2} = \binom{k}{k-1,1,0} a_{1,1}^{k-1} a_{1,2}^1 a_{1,3}^0 =k a_{1,1}^{k-1} a_{1,2}.
\]
To determine $c_{1,3}$ we require $k_{1,2}+2 k_{1,3}=2$ which implies that either $k_{1,3}=1$ and $k_{1,2}=0$ or $k_{1,3}=0$ and $k_{1,2}=2$. The first choice gives
\[
\binom{k}{k-1,0,1} a_{1,1}^{k-1} a_{1,2}^0 a_{1,3}^1 =k a_{1,1}^{k-1} a_{1,3}
\]
and the second choice gives
\[
\binom{k}{k-2,2,0} a_{1,1}^{k-2} a_{1,2}^2 a_{1,3}^0 = \frac{k(k-1)}{2} a_{1,1}^{k-2} a_{1,2}^2.
\]
Adding the two together gives
\[
c_{1,3}=k a_{1,1}^{k-1} a_{1,3}+ \frac{k(k-1)}{2} a_{1,1}^{k-2} a_{1,2}^2.
\]

Take now $B_2$, $B_3$, both $3\times 3$ matrices such that the tuple $(B_1,B_2,B_3)$ is commuting. From the discussion in section 2 and since $B_1$ has only one
eigenvalue, it follows that $A_j=P^{-1}B_jP$ is triangular and Toeplitz for every $j=1,2,3$. The elements of the matrices $A_1$, $A_2$, $A_3$ are defined as in
section 2. Then for a multi-index $k=(k_1,k_2,k_3)$, $k_1,k_2,k_3\in \mathbb{N}$ define
\[
B(k)=\left(
\begin{array} {lll}
b_1(k) & b_2(k) & b_3(k)\\
0 & b_1(k) & b_2(k)\\
0 & 0& b_1(k)
\end{array}
\right) =A_1^{k_1}A_2^{k_2}A_3^{k_3} .
\]
We are now ready to follow the scheme developed in section 4. The proofs of the following lemmata and proposition are similar to that in section 4 and are left to the
interested reader.
\begin{lemma} \label{coef3}
Denoting by $k=(k_1,k_2,k_3)$ a multi-index, $k_1,k_2,k_3\in \mathbb{N}$, we have
\begin{eqnarray*}
b_1(k) &=& a_{1,1}^{k_1} a_{2,1}^{k_2} a_{3,1}^{k_3}\\
b_2(k) &=& b_1(k) \sum_{j=1}^{3} k_j\frac{a_{j,2}}{a_{j,1}}\\
b_3(k) &=& b_1(k) \bigg{(} \sum_{j=1}^{3} k_j (\frac{a_{j,3}}{a_{j,1}} - \frac{1}{2} \frac{a_{j,2}^2}{a_{j,1}^2}) + \frac{1}{2}(\sum_{j=1}^{3}
k_j\frac{a_{j,2}}{a_{j,1}})^2
            \bigg{)}\\
    &=& b_1(k) \bigg{(} \sum_{j=1}^{3} k_j (\frac{a_{j,3}}{a_{j,1}} - \frac{1}{2} \frac{a_{j,2}^2}{a_{j,1}^2}) + \frac{1}{2}\frac{b_2(k)^2}{b_1(k)^2} \bigg{)}
\end{eqnarray*}
\end{lemma}

\begin{lemma} \label{equivhyp3}
The tuple $(A_1,A_2,A_3)$ is hypercyclic if and only if the set
\[
\left\{\left(
\begin{array}{c}
\sum_{j=1}^{3} k_j (\frac{a_{j,3}}{a_{j,1}} - \frac{1}{2} \frac{a_{j,2}^2}{a_{j,1}^2})  \\
\sum_{j=1}^{3} k_j\frac{a_{j,2}}{a_{j,1}}\\
a_{1,1}^{k_1} a_{2,1}^{k_2} a_{3,1}^{k_3}
\end{array}
\right) : k_1, k_2, k_3\in \mathbb{N}\right\}
\]
is dense in $\mathbb{C}^3$.
\end{lemma}

\begin{proposition}
The tuple $(A_1,A_2,A_3)$ is not hypercyclic.
\end{proposition}

\section{4$\times$4 matrices}
Again, as in the previous section, we only deal with $4\times 4$ triangular Toeplitz matrices. So, let $A_1,A_2,A_3,A_4$ be $4\times 4$ triangular Toeplitz matrices
with elements defined as in section 2. For a multi-index $k=(k_1,k_2,k_3,k_4)$, $k_1,k_2,k_3,k_4\in \mathbb{N}$ we compute the entries of the $4\times 4$ matrix
\[
C(k)=\left(
\begin{array} {llll}
c_1(k) & c_2(k) & c_3(k) & c_4(k)\\
0 & c_1(k) & c_2(k) & c_3(k)\\
0 & 0 & c_1(k) & c_2(k)\\
0 & 0 & 0 & c_1(k)
\end{array}
\right) :=A_1^{k_1} A_2^{k_2} A_3^{k_3} A_4^{k_4} ,
\]
since it will be of use to us in the sequel. To determine $c_1(k)$ we require
\[
k_{1,2}+\ldots + k_{4,2}+2(k_{1,3}+\ldots+k_{4,3})+3(k_{1,4}+\ldots+k_{4,4})=0
\]
which implies that all of the above $k$'s are $0$ which implies that $k_{1,1}=k_1$, $k_{2,1}=k_2$, $k_{3,1}=k_3$, $k_{4,1}=k_4$. As a result
\[
c_1(k)=a_{1,1}^{k_1} a_{2,1}^{k_2} a_{3,1}^{k_3} a_{4,1}^{k_4}.
\]
To determine $c_2(k)$ we require
\[
k_{1,2}+\ldots + k_{4,2}+2(k_{1,3}+\ldots+k_{4,3})+3(k_{1,4}+\ldots+k_{4,4})=1
\]
which is the case if one of $k_{1,2}$, $k_{2,2}$, $k_{3,2}$, $k_{4,2}$ is $1$ and all the rest are $0$. Doing the calculation gives
\[
c_2(k)=c_1(k) \sum_{j=1}^{4} k_j\frac{a_{j,2}}{a_{j,1}}.
\]
To determine $c_3(k)$ we require
\[
k_{1,2}+\ldots + k_{4,2}+2(k_{1,3}+\ldots+k_{4,3})+3(k_{1,4}+\ldots+k_{4,4})=2
\]
which is the case if either one of $k_{1,3}$, $k_{2,3}$, $k_{3,3}$, $k_{4,3}$ is 1 or any one of $k_{1,2}$, $k_{2,2}$, $k_{3,2}$, $k_{4,2}$ is $2$ or any two of
$k_{1,2}$, $k_{2,2}$, $k_{3,2}$, $k_{4,2}$ are $1$ and the rest are $0$. Doing the calculation gives
\[
c_3(k) = c_1(k) \left( \sum_{j=1}^{4} k_j \frac{a_{j,3}}{a_{j,1}} + \sum_{l,j=1,\,\, l>j}^{4} k_lk_j \frac{a_{l,2}a_{j,2}}{a_{l,1}a_{j,1}} + \sum_{j=1}^{4}
\frac{k_j(k_j-1)}{2} \frac{a_{j,2}^2}{a_{j,1}^2} \right).
\]
To determine $c_4(k)$ we require
\[
k_{1,2}+\ldots + k_{4,2}+2(k_{1,3}+\ldots+k_{4,3})+3(k_{1,4}+\ldots+k_{4,4})=3
\]
which is the case if either one of $k_{1,4}$, $k_{2,4}$, $k_{3,4}$, $k_{4,4}$ is $1$ or one of $k_{1,3}, k_{2,3}, k_{3,3}, k_{4,3}$ is $1$ and one of $k_{1,2},
k_{2,2}, k_{3,2}, k_{4,2}$ is $1$ or from the set $k_{1,2}, k_{2,2}, k_{3,2}, k_{4,2}$ either one is $3$ or one is $2$ and another is $1$ or any three are $1$. These
cover all the configurations which satisfy the above expression. Doing the calculations gives
\begin{eqnarray*}
c_4(k) &=& c_1(k) \bigg( \sum_{j=1}^{4} k_j \frac{a_{j,4}}{a_{j,1}} + \sum_{j=1}^{4} k_j(k_j-1) \frac{a_{j,2}a_{j,3}}{a_{j,1}^2}\\
       &+& \sum_{j,l=1,\,\, l\neq j}^{4} k_lk_j\frac{a_{j,2}}{a_{j,1}} \frac{a_{l,3}}{a_{l,1}}
    + \sum_{\substack{j,l,p=1 \\ j\neq l,\, j\neq p,\, l\neq p}}^{4} k_jk_lk_p \frac{a_{j,2}a_{l,2}a_{p,2}}{a_{j,1}a_{l,1}a_{p,1}}\\
       &+& \sum_{j=1}^{4} \frac{k_j(k_j-1)(k_j-2)}{6}
         \frac{a_{j,2}^3}{a_{j,1}^3}
    + \sum_{j,l=1,\,\, j\neq l}^{4} \frac{k_j(k_j-1)k_l}{2} \frac{a_{j,2}^2a_{l,2}}{a_{j,1}^2a_{l,1}} \bigg).
\end{eqnarray*}
As in the previous section we do not give proofs of the lemmata and proposition that follow. The proofs are in the spirit of the proofs given in section 4. From the
above calculations plus a little work we get the following
\begin{lemma} \label{coef4}
\begin{eqnarray*}
c_1(k) &=& a_{1,1}^{k_1} a_{2,1}^{k_2} a_{3,1}^{k_3} a_{4,1}^{k_4}\\
c_2(k) &=& c_1(k) \sum_{j=1}^{4} k_j\frac{a_{j,2}}{a_{j,1}}\\
c_3(k) &=& c_1(k) \bigg{(} \sum_{j=1}^{4} k_j (\frac{a_{j,3}}{a_{j,1}} - \frac{1}{2} \frac{a_{j,2}^2}{a_{j,1}^2}) + \frac{1}{2}(\sum_{j=1}^{4}
k_j\frac{a_{j,2}}{a_{j,1}})^2
            \bigg{)}\\
    &=& c_1(k) \bigg{(} \sum_{j=1}^{4} k_j (\frac{a_{j,3}}{a_{j,1}} - \frac{1}{2} \frac{a_{j,2}^2}{a_{j,1}^2}) + \frac{1}{2}\frac{c_2^2}{c_1(k)^2} \bigg{)}\\
c_4(k) &=& c_1(k) \bigg{(} \sum_{j=1}^{4} k_j (\frac{a_{j,4}}{a_{j,1}} - \frac{a_{j,2}a_{j,3}}{a_{j,1}^2} + \frac{1}{3}\frac{a_{j,2}^3}{a_{j,1}^3})\\
    &+& (\sum_{j=1}^{4}k_j \frac{a_{j,2}}{a_{j,1}}) (\sum_{j=1}^{4}k_j (\frac{a_{j,3}}{a_{j,1}} - \frac{1}{2} \frac{a_{j,2}^2}{a_{j,1}^2})) + \frac{1}{6 } (
         \sum_{j=1}^{4} k_j\frac{a_{j,2}}{a_{j,1}})^3 \bigg{)}\\
    &=& c_1(k) \bigg{(} \sum_{j=1}^{4} k_j (\frac{a_{j,4}}{a_{j,1}} - \frac{a_{j,2}a_{j,3}}{a_{j,1}^2} + \frac{1}{3}\frac{a_{j,2}^3}{a_{j,1}^3})\\
    &+& \frac{c_2(k)}{c_1(k)}(\frac{c_3(k)}{c_1(k)}-\frac{1}{2}\frac{c_2(k)^2}{c_1(k)^2})+\frac{1}{6}\frac{c_2(k)^3}{c_1(k)^3}\bigg{)}
\end{eqnarray*}
\end{lemma}

\begin{lemma} \label{equivhyp4}
The tuple $(A_1,A_2,A_3,A_4)$ is hypercyclic if and only if the set
\[
\left\{\left(
\begin{array}{c}
\sum_{j=1}^{4} k_j (\frac{a_{j,4}}{a_{j,1}} - \frac{a_{j,2}a_{j,3}}{a_{j,1}^2} + \frac{1}{3}\frac{a_{j,2}^3}{a_{j,1}^3})\\
\sum_{j=1}^{4} k_j (\frac{a_{j,3}}{a_{j,1}} - \frac{1}{2} \frac{a_{j,2}^2}{a_{j,1}^2})  \\
\sum_{j=1}^{4} k_j\frac{a_{j,2}}{a_{j,1}}\\
a_{1,1}^{k_1} a_{2,1}^{k_2} a_{3,1}^{k_3} a_{4,1}^{k_4}
\end{array}
\right) : k_1, k_2, k_3, k_4\in \mathbb{N}\right\}
\]
is dense in $\mathbb{C}^4$.
\end{lemma}

\begin{proposition}
The tuple $(A_1,A_2,A_3,A_4)$ is not hypercyclic.
\end{proposition}

\section{Open problems}

\noindent \textbf{Problem 1}. What is the minimal number of complex upper triangular Toeplitz matrices that form a hypercyclic tuple in $\mathbb{C}^n$? Observe that
in \cite{CoHaMa} the authors have shown the existence of an eight tuple of complex upper triangular $2\times 2$ Toeplitz matrices being hypercyclic in $\mathbb{C}^2$,
which is far from being optimal.

\medskip

\noindent \textbf{Problem 2}. What is the minimal number of real upper triangular Toeplitz matrices that form a hypercyclic tuple in $\mathbb{R}^n$? In \cite{CoPa}
the authors have shown that the minimal number for $2\times 2$ matrices is $3$.

\medskip

\noindent \textbf{Problem 3}. In \cite{CoHaMa} an example of a two-tuple $(A_1,A_2)$ of commuting $2\times 2$ matrices is given in which every non-zero vector in
$\mathbb{R}^2$ is hypercyclic for $(A_1,A_2)$. In other words we say that the $\mathbb{N}^2$ action is minimal. Does the above extend in any dimension $n\geq 3$? That
is, for $n\geq 3$, does there exist a hypercyclic $n$-tuple $(A_1, \ldots ,A_n)$ of commuting $n\times n$ matrices for which every non-zero vector $x$ in
$\mathbb{R}^n$ is hypercyclic?

\end{document}